\documentclass[12pt]{article}
\usepackage[margin=2cm]{geometry}
\usepackage{graphicx}
\usepackage{tikz}
\usepackage[numbers]{natbib}
\usepackage{amsmath}
\usepackage{amsthm}
\usepackage{amssymb}
\usepackage{hyperref}
\usepackage[capitalise,nameinlink]{cleveref}

\newcommand{\graphclass}{\mathcal{G}}   
\newcommand{\vertexset}{V}              
\newcommand{\edgeset}{E}                
\newcommand{\graphdist}{d}              
\newcommand{\seqset}{\mathcal{S}}       
\newcommand{\egseq}{\sigma}             

\newtheorem{theorem}{Theorem}[section]

\theoremstyle{plain}
\newtheorem{lemma}[theorem]{Lemma}

\newtheorem{proposition}[theorem]{Proposition}

\theoremstyle{definition}
\newtheorem{definition}[theorem]{Definition}

\theoremstyle{remark}
\newtheorem{remark}[theorem]{Remark}

\title{On the Realizability of Edge-Girth Sequences}

\author{
Lilian Marey\thanks{LTCI, Télécom Paris, Institut Polytechnique de Paris, Palaiseau, France (\texttt{lilian.marey@telecom-paris.fr}).}
\and
Paul Hilaire\thanks{LTCI, Inria, Télécom Paris, Institut Polytechnique de Paris, Palaiseau, France (\texttt{paul.hilaire@telecom-paris.fr}).}
\and
Charlotte Laclau\thanks{LTCI, Télécom Paris, Institut Polytechnique de Paris, Palaiseau, France (\texttt{charlotte.laclau@telecom-paris.fr}).}
}
\date{}

\begin{document}

\maketitle

\begin{abstract}
The \emph{edge-girth} of an edge $e$ in a simple connected graph is the
length of a shortest cycle containing $e$, with $g_e = \infty$ if no such
cycle exists, and the \emph{edge-girth sequence} of a graph is the
nondecreasing sequence of the edge-girths of its edges. We characterize the
sequences that are realizable as the edge-girth sequence of a simple connected
graph. A constant sequence $(g^{(m)})$ with $g$ finite and $m \geq 1$ is
realizable if and only if $m = g$ or $m \geq \lceil 3g/2 \rceil$. Otherwise,
writing $S = S_0 \uplus (g^{(m)})$, where $g$ is the maximum value of $S$ and
$m$ its multiplicity, we prove that $S$ is realizable if and only if $S_0$ is
realizable and $m$ lies in an explicit set determined by $g$ and by the
maximum diameters of the graphs realizing $S_0$, alone or extended by fewer
than $g$ edges of edge-girth $g$. The proof rests on a chaining lemma, which
turns a shortest cycle of a realization into a lower bound on the maximum
diameter of the realizations of a shorter sequence. 
\end{abstract}


\section{Introduction}
\label{sec:c5:intro}

The \emph{graph realization problem} asks whether some graph matches a
prescribed sequence of invariants, and how to construct one if so. Its most
celebrated instance is the Erd\H{o}s--Gallai theorem~\cite{erdhos1960graphs},
which characterizes the degree sequences of simple graphs, with constructive
proofs due to Havel~\cite{havel1955poznamka} and
Hakimi~\cite{hakimi1962realizability}. The problem has since been extended to
richer degree-based data, such as joint degree
distributions~\cite{stanton2012constructing} and neighborhood degree
lists~\cite{barrus2018neighborhood}. Invariants built on cycles have mostly
been studied through the girth: cages, the smallest regular graphs of given
degree and girth~\cite{exoo2012dynamic}, and edge-girth-regular graphs, in
which every edge of a regular graph lies on the same number of shortest
cycles~\cite{jajcay2018edge, goedgebeur2025exhaustive}, together with their
vertex analogue~\cite{jajcay2025vertex}. Cycles in random graphs have also
been studied at length~\cite{mckay2004short, lopez2018exactly, alon2022cycle}.

In this paper, we study the \emph{edge-girth} $g_e$ of an edge $e$, defined
as the length of a shortest cycle containing $e$, with $g_e = \infty$ if no
such cycle exists. We define the \emph{edge-girth sequence} $\sigma(G)$ of a
graph $G$ as the nondecreasing sequence of edge-girth values over all edges
(\Cref{fig:c5:edge-girth-seq}), and ask: which sequences are realizable?
This invariant is strictly finer than the \emph{girth} $\min_{e \in E} g_e$:
two graphs may share the same girth while having entirely different
edge-girth sequences. It also differs from edge-girth regularity, since the
length of a shortest cycle may vary from edge to edge and no regularity is
assumed.

The edge-girth arises in several contexts. It equals one plus the length of a
shortest path between the endpoints of $e$ that avoids $e$, the quantity
computed by the replacement path problem~\cite{bernstein2010nearly}. In flow
networks, it indexes transition rates in models of stochastic cycle selection
in active flows~\cite{woodhouse2016stochastic}. In coding theory, short cycles
in the Tanner graph of a low-density parity-check code degrade iterative
decoding, and the relevant quantity is the length of the shortest cycles
through individual edges~\cite{xu2025ldpc}.

The characterizations of degree sequences are recursive and local: a sequence
is reduced to a shorter one whose realizability is equivalent, through an
exchange that only modifies edges incident to a bounded set of vertices. The
recursion carries over to edge-girth sequences, and our characterization
reduces a sequence to its prefix below the maximum value. Locality does not:
the edge-girth of an edge can be changed by a modification of the graph
arbitrarily far from it. We therefore replace local exchanges by three
operations acting on the graph at large (identifying a vertex of one graph with
a vertex of another, removing edges, and attaching a path), whose effect on
the edge-girth of every edge is controlled. The characterization they yield
involves a global quantity, the largest diameter of the graphs realizing a
sequence.

Our main result, \Cref{thm:c5:main}, characterizes the realizable
edge-girth sequences recursively. Writing $S = S_0 \uplus (g^{(m)})$ with
$g = \max(S)$, the sequence $S$ is realizable if and only if $S_0$ is, and $m$
lies in an explicit set determined by $g$ and by the maximum diameters of the
graphs realizing $S_0$, alone or extended by fewer than $g$ edges of
edge-girth $g$. The key tool is a chaining lemma (\Cref{lem:c5:chain}):
when the edges of maximum edge-girth are removed, a shortest cycle through
some of them falls apart into arcs that can be laid end to end in a
realization of a shorter sequence, which bounds its maximum diameter from
below.

The paper is organized as follows. \Cref{sec:c5:notation} introduces
notation and the properties relating edge-girth to distance.
\Cref{sec:c5:operations} develops the graph operations used throughout.
\Cref{sec:c5:base} treats sequences with infinite values and constant
sequences, \Cref{sec:c5:induction} proves the chaining lemma and
characterizes the admissible multiplicities, and \Cref{sec:c5:main}
states the main result. \Cref{sec:c5:conclusion} discusses open
questions, starting with the computation of the maximum diameter.

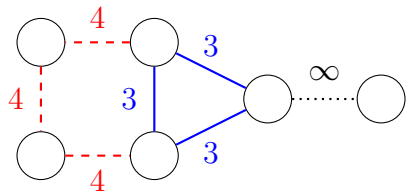
\begin{figure}[htbp]
\centering
\begin{tikzpicture}[
  every node/.style={circle, draw, fill=white, inner sep=2pt, minimum size=18pt},
]

\node (t1) at (0, 1.5)  {};
\node (t2) at (0, 0)    {};
\node (t3) at (1.5, 0.75) {};

\draw[thick, blue] (t1) -- node[draw=none, fill=none, left,  blue] {$3$} (t2);
\draw[thick, blue] (t2) -- node[draw=none, fill=none, below, blue] {$3$} (t3);
\draw[thick, blue] (t1) -- node[draw=none, fill=none, above, blue] {$3$} (t3);

\node (q1) at (-1.5, 1.5) {};
\node (q2) at (-1.5, 0)   {};

\draw[thick, dashed, red] (t1) -- node[draw=none, fill=none, above, red] {$4$} (q1);
\draw[thick, dashed, red] (q1) -- node[draw=none, fill=none, left,  red] {$4$} (q2);
\draw[thick, dashed, red] (q2) -- node[draw=none, fill=none, below, red] {$4$} (t2);

\node (b1) at (3, 0.75) {};
\draw[thick, dotted, black] (t3) --
  node[draw=none, fill=none, above] {$\infty$} (b1);


\end{tikzpicture}
\caption{A graph $G$ with edge-girth sequence
  $\sigma(G) = \bigl(3^{(3)},\, 4^{(3)},\, \infty^{(1)}\bigr)$.
  Solid edges have edge-girth~$3$, dashed edges have edge-girth~$4$,
  and the dotted edge is a bridge ($g_e = \infty$).}
\label{fig:c5:edge-girth-seq}
\end{figure}

\section{Preliminaries}
\label{sec:c5:notation}

We introduce the notation and basic definitions used throughout the paper,
and establish the properties relating edge-girth to shortest-path distance on
which the subsequent proofs rest.

\subsection{Definitions and Notation}

Let $\graphclass$ denote the class of simple connected graphs. For a graph $G$,
we write $\vertexset(G)$ for its vertex set and $\edgeset(G)$ for its edge set,
and, for $F \subseteq \edgeset(G)$, $G \setminus F$ for the graph with vertex
set $\vertexset(G)$ and edge set $\edgeset(G) \setminus F$. For a cycle $C$,
$\edgeset(C)$ denotes its set of edges and $|C|$ its length, and a cycle of
length $g$ is called a \emph{$g$-cycle}. A \emph{walk} of length $\ell$ is a
sequence of vertices $(v_0, \ldots, v_\ell)$ in which consecutive vertices are
adjacent; it \emph{traverses} the edges $\{v_{i-1}, v_i\}$,
$1 \leq i \leq \ell$, counted with repetition, and it is \emph{closed} if
$v_0 = v_\ell$. We write $|W|$ for the length of a walk $W$.

\begin{definition}[Edge-girth]
\label{def:c5:edge-girth}
Let $G$ be a simple graph. For an edge $e \in \edgeset(G)$, the
\emph{edge-girth} of $e$, denoted $g_e$, is the minimum number of edges in a
cycle of $G$ containing $e$. If no such cycle exists, we set $g_e = \infty$.
When the graph must be specified explicitly, we write $g_{e, G}$.
\end{definition}

The definition applies verbatim to a graph that is not connected, such as
$G \setminus F$: the edge-girth of an edge is then the one it has in its
connected component. Note that $g_e \in \{3, 4, 5, \ldots\} \cup \{\infty\}$
for every edge $e$. We define the set of admissible values
\[
  \mathcal{L} \;:=\; \{3,4,5,\ldots\} \cup \{\infty\},
\]
and the set of finite admissible values
$\mathcal{L}^* \;:=\; \mathcal{L} \setminus \{ \infty\}$.

\begin{definition}[Edge-girth sequence]
\label{def:c5:edge-girth-seq}
The \emph{edge-girth sequence} of $G \in \graphclass$ is the nondecreasing
sequence of the values $(g_e)_{e \in \edgeset(G)}$, written as
\[
  \egseq(G) \;=\; \bigl(g_1^{(m_1)},\, \ldots,\, g_k^{(m_k)}\bigr),
\]
where $g_1 < \cdots < g_k$ are the distinct edge-girth values and $m_i$ denotes
the multiplicity of $g_i$, that is, the number of edges of edge-girth $g_i$.
\end{definition}

\begin{definition}[Sequence of admissible values]
\label{def:c5:seq-C}
Let $\mathrm{Seq}(\mathcal{L}) := \bigcup_{n \geq 0} \mathcal{L}^n$ denote the set
of all finite sequences over $\mathcal{L}$, including the empty sequence.
For a non-empty sequence $S = (g_1^{(m_1)},\, \ldots,\, g_k^{(m_k)}) \in \mathrm{Seq}(\mathcal{L})$, we write $\max(S) := g_k$ for the maximum edge-girth value of $S$.
\end{definition}

\begin{definition}[Realizable sequence, realization set]
\label{def:c5:realizable}
A sequence $S \in \mathrm{Seq}(\mathcal{L})$ is \emph{realizable} if there exists
$G \in \graphclass$ such that $\egseq(G) = S$. We denote by
\[
  \seqset \;:=\; \bigl\{\, S \in \mathrm{Seq}(\mathcal{L})
    : \exists\, G \in \graphclass,\; \egseq(G) = S \,\bigr\}
\]
the set of all realizable sequences, and by $\seqset^* := \seqset \cap
\mathrm{Seq}(\mathcal{L}^*)$ the subset of those with only finite values.
For $S \in \seqset$ written as $S = (g_1^{(m_1)}, \ldots, g_k^{(m_k)})$, we
set $|S| := \sum_{i=1}^k m_i$. Note that any $G \in \egseq^{-1}(S)$ has exactly
$|S|$ edges. The \emph{realization set} of $S$ is
\[
  \egseq^{-1}(S) \;:=\; \{\, G \in \graphclass : \egseq(G) = S \,\}.
\]
By convention, the empty sequence belongs to $\seqset^*$, its realization set
consisting of the graph with a single vertex, the only edgeless graph of $\graphclass$.
\end{definition}

\begin{definition}[Multiset sum]
\label{def:c5:uplus}
For two sequences $S_1, S_2 \in \mathrm{Seq}(\mathcal{L})$, their \emph{multiset
sum} is the sequence
\[
  S_1 \uplus S_2 \;:=\; \bigl(g^{(m_g(S_1) + m_g(S_2))}\bigr)_{g \in \mathcal{L}},
\]
where $m_g(S)$ denotes the multiplicity of $g$ in $S$, with
$m_g(S) = 0$ if $g$ does not appear in $S$.
\end{definition}

\begin{definition}[Realizing cycle]
\label{def:c5:realizing-cycle}
Let $G \in \graphclass$ and let $C$ be a cycle of $G$. We say that $C$
\emph{realizes} the edge-girth of $e \in \edgeset(C)$ if $|C| = g_e$, and call $C$ a
\emph{realizing cycle} for $e$.
\end{definition}

For $G \in \graphclass$ and $u, v \in \vertexset(G)$, we denote by $\graphdist_G(u,v)$ the
\emph{shortest-path distance} between $u$ and $v$ in $G$, that is, the minimum
number of edges of a path between them, and by
$\mathrm{diam}(G) := \max_{u,v \in \vertexset(G)} \graphdist_G(u,v)$ its \emph{diameter}.

\begin{definition}[Maximum diameter]
\label{def:c5:diam}
Let $S \in \seqset$. The \emph{maximum diameter} of $S$ is
\[
  d^*_S := \max\bigl\{\mathrm{diam}(G) : G \in \egseq^{-1}(S)\bigr\}.
\]
\end{definition}

\begin{remark}
\label{rem:c5:diam-welldefined}
Since $S \in \seqset$, the set $\egseq^{-1}(S)$ is non-empty. Every
$G \in \egseq^{-1}(S)$ has exactly $|S|$ edges, so $\mathrm{diam}(G) \leq |S|$
and the set of diameters is finite. Hence $d^*_S$ is well-defined and attained.
\end{remark}

\subsection{Edge-Girth and Graph Distance}

The following lemma expresses the edge-girth of an edge as a distance in the
graph from which the edge is removed, and bounds it by the length of any closed
walk that traverses the edge exactly once.

\begin{lemma}
\label{lem:c5:closed-walk}
Let $G$ be a simple graph and $e = \{u, v\} \in \edgeset(G)$.
\begin{enumerate}
  \item[\textup{(1)}] $g_e = 1 + \graphdist_{G \setminus \{e\}}(u, v)$, with
    $g_e = \infty$ when $u$ and $v$ are disconnected in $G \setminus \{e\}$.
  \item[\textup{(2)}] If $W$ is a closed walk of $G$ that traverses $e$
    exactly once, then $|W| \geq g_e$.
\end{enumerate}
\end{lemma}

\begin{proof}
\textit{(1).} Removing $e$ from a cycle of length $g_e$ through $e$ leaves a
path of length $g_e - 1$ between $u$ and $v$ in $G \setminus \{e\}$, so that
$\graphdist_{G \setminus \{e\}}(u,v) \leq g_e - 1$. Conversely, a shortest path
between $u$ and $v$ in $G \setminus \{e\}$ has length at least $2$, since $e$
is the only edge joining $u$ and $v$, and closing it with $e$ yields a cycle
through $e$ of length $1 + \graphdist_{G \setminus \{e\}}(u,v)$.

\textit{(2).} Deleting from $W$ its unique traversal of $e$ leaves a walk
between $u$ and $v$ that does not use $e$, of length $|W| - 1$. This walk
contains a path between $u$ and $v$ in $G \setminus \{e\}$, so
$\graphdist_{G \setminus \{e\}}(u,v) \leq |W| - 1$, and~\textup{(1)} gives
$g_e \leq |W|$.
\end{proof}

\begin{proposition}
\label{prop:c5:realizing-shortest}
Let $G \in \graphclass$, $e = \{v_1, v_2\} \in \edgeset(G)$ with $g_e = g < \infty$, and
let $C = (v_1, v_2, \ldots, v_g, v_1)$ be a realizing cycle for $e$.
\begin{enumerate}
  \item[\textup{(1)}] For all $i, j \in \{1, \ldots, g\}$,
    $\graphdist_G(v_i, v_j) = \min(|i-j|,\, g - |i-j|)$.
  \item[\textup{(2)}] For all $k \in \{1, \ldots, \lfloor g/2 \rfloor\}$, there exist
    $i, j \in \{1, \ldots, g\}$ such that $\graphdist_G(v_i, v_j) = k$.
\end{enumerate}
\end{proposition}

\begin{proof}
\textit{(1).} The case $i = j$ is trivial. For $i \neq j$, the cycle $C$ splits
into two paths between $v_i$ and $v_j$, of lengths $|i-j|$ and $g - |i-j|$;
call $A$ the one that contains $e$ and $A'$ the other, so that
$|A| + |A'| = g$. Both are paths of $G$, hence
$\graphdist_G(v_i,v_j) \leq \min(|i-j|, g-|i-j|)$. Suppose that a path $P$
between $v_i$ and $v_j$ satisfies $|P| < \min(|A|, |A'|)$. If $P$ does not
traverse $e$, the closed walk formed by $P$ and $A$ traverses $e$ exactly once
and has length $|P| + |A| < g$. If $P$ traverses $e$, the closed walk formed by
$P$ and $A'$ traverses $e$ exactly once, since $A'$ avoids $e$, and has length
$|P| + |A'| < g$. Both cases contradict \Cref{lem:c5:closed-walk}(2), since
$g_e = g$.

\textit{(2).} By~\textup{(1)}, $\graphdist_G(v_1, v_{k+1}) = \min(k, g-k) = k$ for every
$k \in \{1, \ldots, \lfloor g/2 \rfloor\}$.
\end{proof}

\begin{proposition}
\label{prop:c5:overlap}
Let $G \in \graphclass$, let $C_1$ and $C_2$ be two $g$-cycles of $G$ with
$g < \infty$, and suppose that some edge $e \in \edgeset(C_2) \setminus
\edgeset(C_1)$ has $g_{e} = g$. Then
\[
  |\edgeset(C_1) \cap \edgeset(C_2)| \leq \lfloor g/2 \rfloor.
\]
\end{proposition}

\begin{proof}
Let $s := |\edgeset(C_1) \cap \edgeset(C_2)|$ and
$E_2 := \edgeset(C_2) \setminus \edgeset(C_1)$, so that $|E_2| = g - s$ and
$e \in E_2$. Suppose, for a contradiction, that $s > \lfloor g/2 \rfloor$.
Then $s \geq 1$ and $|E_2| < g - \lfloor g/2 \rfloor = \lceil g/2 \rceil$. We
build a cycle through $e$ from edges of $E_2$ and an arc of $C_1$.

Follow $C_2$ from $e$ in both directions, up to the first vertex of $C_1$ met
on each side. This yields a path $P$ of $C_2$ that contains $e$, whose two ends
$p$ and $q$ lie on $C_1$, and whose internal vertices do not. The ends are
distinct: otherwise $p$ would be the only vertex of $C_2$ on $C_1$, whereas the
two ends of a shared edge, which exists since $s \geq 1$, lie on both cycles.
No edge of $P$ lies on $C_1$: this holds for $e$, and every other edge of $P$
has an end inside $P$, hence off $C_1$. Thus $P$ consists of edges of $E_2$,
and $|P| < \lceil g/2 \rceil$.

The cycle $C_1$ is the union of two paths between $p$ and $q$; let $A$ be the
shorter one, so that $|A| \leq \lfloor g/2 \rfloor$. Since the internal
vertices of $P$ are not on $C_1$, the paths $P$ and $A$ form a cycle through
$e$, of length $|P| + |A| < \lceil g/2 \rceil + \lfloor g/2 \rfloor = g$,
which contradicts $g_e = g$.
\end{proof}


\section{Structural Operations and Their Effect on Edge-Girth}
\label{sec:c5:operations}

In this section, we define three graph operations --- vertex identification, edge removal and path attachment --- and study how each affects the edge-girth sequence.
The first two are classical; what is established here is their effect on the invariant.
They are the building blocks of all the constructions and proofs of
\Cref{sec:c5:base,sec:c5:induction}.

\begin{definition}[Vertex identification]
\label{def:c5:vertex-id}
Let $G_1, G_2 \in \graphclass$ have disjoint vertex sets, and let $v_i \in \vertexset(G_i)$ for $i = 1, 2$.
The \emph{vertex identification} of $G_1$ and $G_2$ at $(v_1, v_2)$ (\Cref{fig:c5:vertex-id}) is the graph $H = G_1 \underset{v_1,v_2}{\oplus} G_2$ with
\[
  \vertexset(H) = \bigl(\vertexset(G_1) \cup \vertexset(G_2)\bigr) \setminus \{v_2\}
\]
and
\[
  \edgeset(H) = \edgeset(G_1) \;\cup\; \bigl\{\{v_1, v\} : \{v_2, v\} \in \edgeset(G_2)\bigr\}
         \;\cup\; \bigl(\edgeset(G_2) \setminus \bigl\{\{v_2,v\} : v \in \vertexset(G_2)\bigr\}\bigr).
\]
That is, the two graphs are glued at the single vertex $v_1$, onto which $v_2$ is merged.
\end{definition}

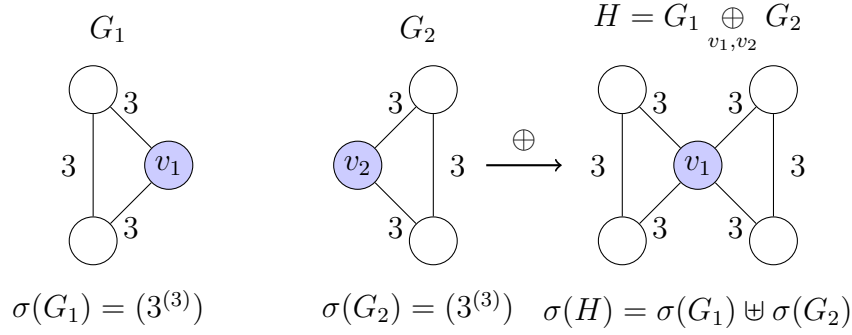
\begin{figure}[htbp]
\centering
\begin{tikzpicture}[
  every node/.style={circle, draw, fill=white, inner sep=2pt, minimum size=18pt},
  lbl/.style={draw=none, fill=none},
]

\node (a1) at (0, 1)  {};
\node (a2) at (0, -1) {};
\node[fill=blue!20] (v1) at (1, 0)  {$v_1$};
\draw[] (a1) -- node[draw=none, fill=none, above] {$3$} (v1);
\draw[] (a2) -- node[draw=none, fill=none, below] {$3$} (v1);
\draw[] (a1) -- node[draw=none, fill=none, left] {$3$} (a2);
\node[draw=none, fill=none] at (0.2, 1.8) {$G_1$};
\node[draw=none, fill=none, below] at (0.2, -.5) {$\sigma(G_1) = (3^{(3)})$};

\node[fill=blue!20] (v2) at (3.5, 0) {$v_2$};
\node (b1) at (4.5,  1) {};
\node (b2) at (4.5, -1) {};
\draw[] (v2) -- node[draw=none, fill=none, above] {$3$} (b1);
\draw[] (v2) -- node[draw=none, fill=none, below] {$3$} (b2);
\draw[] (b1) -- node[draw=none, fill=none, right] {$3$} (b2);
\node[draw=none, fill=none] at (4.3, 1.8) {$G_2$};
\node[draw=none, fill=none, below] at (4.3, -.5) {$\sigma(G_2) = (3^{(3)})$};
\draw[->, thick] (5.2, 0) -- (6.2, 0)
  node[draw=none, fill=none, midway, above] {$\oplus$};

\node (c1) at (7,  1)  {};
\node (c2) at (7, -1)  {};
\node[fill=blue!20] (v) at (8, 0) {$v_1$};
\node (d1) at (9,  1)  {};
\node (d2) at (9, -1)  {};
\draw[] (c1) -- node[draw=none, fill=none, above] {$3$} (v);
\draw[] (c2) -- node[draw=none, fill=none, below] {$3$} (v);
\draw[] (c1) -- node[draw=none, fill=none, left] {$3$} (c2);
\draw[] (v) -- node[draw=none, fill=none, above] {$3$} (d1);
\draw[] (v) -- node[draw=none, fill=none, below] {$3$} (d2);
\draw[] (d1) -- node[draw=none, fill=none, right] {$3$} (d2);
\node[draw=none, fill=none] at (8, 1.8) {$H = G_1 \underset{v_1,v_2}{\oplus} G_2$};
\node[draw=none, fill=none, below] at (8, 0.2) {$\sigma(H) = \sigma(G_1) \uplus \sigma(G_2)$};

\end{tikzpicture}
\caption[Vertex identification of $G_1$ and $G_2$ at $(v_1, v_2)$]{Vertex identification of $G_1$ and $G_2$ at $(v_1, v_2)$.%
\label{fig:c5:vertex-id}}
\end{figure}

\begin{proposition}
\label{prop:c5:concat}
Let $G_1, G_2 \in \graphclass$ have disjoint vertex sets, let $v_i \in \vertexset(G_i)$ for $i = 1, 2$, and let $H = G_1 \underset{v_1,v_2}{\oplus} G_2$, in which every vertex of $G_2$ keeps its name except $v_2$, which becomes $v_1$.
\begin{enumerate}
  \item[\textup{(1)}] $\egseq(H) = \egseq(G_1) \uplus \egseq(G_2)$. In particular, $\seqset$ and $\seqset^*$ are closed under $\uplus$.
  \item[\textup{(2)}] The distance in $H$ between two vertices of the same $G_i$ is their distance in $G_i$, and $\graphdist_H(x, y) = \graphdist_{G_1}(x, v_1) + \graphdist_{G_2}(v_2, y)$ for all $x \in \vertexset(G_1)$ and $y \in \vertexset(G_2)$.
\end{enumerate}
\end{proposition}

\begin{proof}
Every path of $H$ between a vertex of $G_1$ and a vertex of $G_2$ passes through $v_1$.

\textit{(1).} A cycle of $H$ passes at most once through $v_1$, so it lies
entirely within $G_1$ or entirely within $G_2$. Therefore $g_{e, H} = g_{e,G_1}$
for every $e \in \edgeset(G_1)$, and likewise for the edges of $G_2$, which
gives $\egseq(H) = \egseq(G_1) \uplus \egseq(G_2)$. Identifying realizations of
two realizable sequences at arbitrary vertices thus realizes their multiset sum.

\textit{(2).} A path of $H$ between two vertices of $G_1$ that left $G_1$ would
pass twice through $v_1$, so shortest paths between vertices of $G_1$ stay in
$G_1$, and likewise for $G_2$. A path between $x \in \vertexset(G_1)$ and
$y \in \vertexset(G_2)$ is a path from $x$ to $v_1$ in $G_1$ followed by a path
from $v_1$ to $y$ in $G_2$, which gives the formula.
\end{proof}

\begin{definition}[Edge removal]
\label{def:c5:edge-removal}
Let $G \in \graphclass$ and $F \subseteq \edgeset(G)$.
The graph obtained by \emph{removing} $F$ from $G$ (\Cref{fig:c5:edge-removal}) is $G \setminus F$.
When $G \setminus F$ is disconnected, its connected components can be rejoined by successive vertex identifications (\Cref{def:c5:vertex-id}), which yields a graph of $\graphclass$ in which, by \Cref{prop:c5:concat}(1), every edge has the edge-girth it has in $G \setminus F$.
\end{definition}

\begin{figure}[htbp]
\centering
\begin{tikzpicture}[
  every node/.style={circle, draw, fill=white, inner sep=2pt, minimum size=18pt},
  lbl/.style={draw=none, fill=none},
]

\begin{scope}[xshift=0cm]
  \node (t1) at (0, 1.5)    {};
  \node (t2) at (0, 0)      {};
  \node (t3) at (1.5, 0.75) {};
  \node (q1) at (-1.5, 1.5) {};
  \node (q2) at (-1.5, 0)   {};

  \draw[thick, blue] (t1) -- node[lbl, left,  blue]{$3$} (t2);
  \draw[thick, blue] (t2) -- node[lbl, below, blue]{$3$} (t3);
  \draw[thick, blue] (t1) -- node[lbl, above, blue]{$3$} (t3);
  \draw[thick, dashed, red]  (t1) -- node[lbl, above, red]{$4$}  (q1);
  \draw[thick, dashed, red]  (q1) -- node[lbl, left,  red]{$4$}  (q2);
  \draw[thick, dashed, red]  (q2) -- node[lbl, below, red]{$4$}  (t2);

  \node[lbl] at (0, -1) {$G$};
\end{scope}

\draw[->, thick] (2.3, 0.75) -- (3.5, 0.75)
  node[lbl, midway, above, font=\small]{remove $E_{\max}$};

\begin{scope}[xshift=5cm]
  \node (t1) at (0, 1.5)    {};
  \node (t2) at (0, 0)      {};
  \node (t3) at (1.5, 0.75) {};
  \draw[thick, blue] (t1) -- node[lbl, left,  blue]{$3$} (t2);
  \draw[thick, blue] (t2) -- node[lbl, below, blue]{$3$} (t3);
  \draw[thick, blue] (t1) -- node[lbl, above, blue]{$3$} (t3);

  \node[lbl] at (0, -1.4) {$G \setminus E_{\max}$};
\end{scope}

\end{tikzpicture}
\caption[Removing the edges of maximum edge-girth]{Removing $\edgeset_{\max}$ (dashed edges, $g_e = 4$) from $G$. The solid edges ($g_e = 3$) are unaffected.
\label{fig:c5:edge-removal}}
\end{figure}
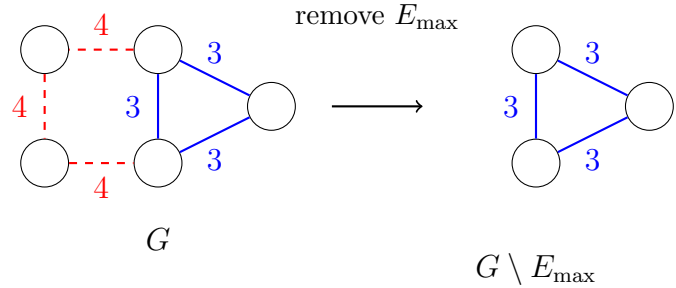

\begin{proposition}
\label{prop:c5:remove-girth}
Let $G \in \graphclass$ and $F \subseteq \edgeset(G)$.
\begin{enumerate}
  \item[\textup{(1)}] Every edge $e \in \edgeset(G) \setminus F$ that has a
    realizing cycle disjoint from $F$ satisfies $g_{e, G \setminus F} = g_{e, G}$.
  \item[\textup{(2)}] Let $S = \egseq(G)$ be non-empty and
    $\edgeset_{\max} = \{ e \in \edgeset(G) : g_e = \max(S) \}$. Every realizing
    cycle of an edge $e \in \edgeset(G) \setminus \edgeset_{\max}$ is disjoint
    from $\edgeset_{\max}$; hence, if $F \subseteq \edgeset_{\max}$, then
    $g_{e, G \setminus F} = g_{e, G}$ for every
    $e \in \edgeset(G) \setminus \edgeset_{\max}$.
\end{enumerate}
In both cases, these edge-girths are unchanged when the components of
$G \setminus F$ are rejoined by vertex identifications.
\end{proposition}

\begin{proof}
\textit{(1).} Removing edges creates no cycle, so
$g_{e, G \setminus F} \geq g_{e, G}$, and the realizing cycle of $e$ disjoint
from $F$ survives in $G \setminus F$, so $g_{e, G \setminus F} \leq g_{e, G}$.

\textit{(2).} Let $e \notin \edgeset_{\max}$, so that $g_e < \max(S)$, and let
$C$ be a realizing cycle of $e$. Every edge $f$ of $C$ satisfies
$g_f \leq |C| = g_e < \max(S)$, so $f \notin \edgeset_{\max}$. The conclusion
follows from~\textup{(1)}.

The last statement is \Cref{prop:c5:concat}(1).
\end{proof}

\begin{definition}[Path attachment]
\label{def:c5:path-attachment}
Let $G \in \graphclass$, $v, v' \in \vertexset(G)$, and $k \geq 1$. The graph
$\Gamma^k_{v,v'}(G)$ (\Cref{fig:c5:path-attachment}) is obtained from $G$ by
adding a path of $k$ new edges between $v$ and $v'$, whose $k-1$ internal
vertices are new; when $v = v'$, the added path closes into a cycle of
length $k$ through $v$.
\end{definition}

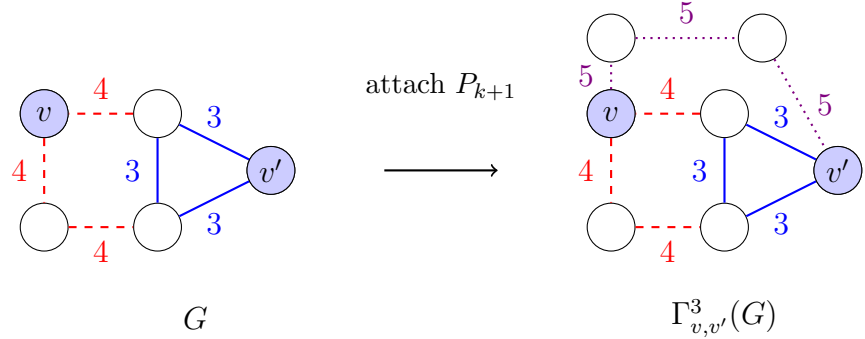
\begin{figure}[htbp]
\centering
\begin{tikzpicture}[
  every node/.style={circle, draw, fill=white, inner sep=2pt, minimum size=18pt},
  lbl/.style={draw=none, fill=none},
]

\begin{scope}[xshift=0cm]
  \node (t1) at (0, 1.5)    {};
  \node (t2) at (0, 0)      {};
  \node (t3) at (1.5, 0.75) {};
  \draw[thick, blue] (t1) -- node[lbl, left,  blue]{$3$} (t2);
  \draw[thick, blue] (t2) -- node[lbl, below, blue]{$3$} (t3);
  \draw[thick, blue] (t1) -- node[lbl, above, blue]{$3$} (t3);

  \node (q1) at (-1.5, 1.5) {};
  \node (q2) at (-1.5, 0)   {};
  \draw[thick, dashed, red] (t1) -- node[lbl, above, red]{$4$} (q1);
  \draw[thick, dashed, red] (q1) -- node[lbl, left,  red]{$4$} (q2);
  \draw[thick, dashed, red] (q2) -- node[lbl, below, red]{$4$} (t2);

  \node[fill=blue!20] at (q1) [label=center:$v$]  {};
  \node[fill=blue!20] at (t3) [label=center:$v'$] {};

  \node[lbl] at (0.5, -1.2){$G$};
\end{scope}

\draw[->, thick] (3.0, 0.75) -- (4.5, 0.75)
  node[lbl, midway, above, font=\small]{attach $P_{k+1}$};

\begin{scope}[xshift=5.5cm]
  \node (t1) at (2, 1.5)    {};
  \node (t2) at (2, 0)      {};
  \node (t3) at (3.5, 0.75) {};
  \draw[thick, blue] (t1) -- node[lbl, left,  blue]{$3$} (t2);
  \draw[thick, blue] (t2) -- node[lbl, below, blue]{$3$} (t3);
  \draw[thick, blue] (t1) -- node[lbl, above, blue]{$3$} (t3);

  \node (q1) at (.5, 1.5) {};
  \node (q2) at (.5, 0)   {};
  \draw[thick, dashed, red] (t1) -- node[lbl, above, red]{$4$} (q1);
  \draw[thick, dashed, red] (q1) -- node[lbl, left,  red]{$4$} (q2);
  \draw[thick, dashed, red] (q2) -- node[lbl, below, red]{$4$} (t2);

  \node[fill=blue!20] at (q1) [label=center:$v$]  {};
  \node[fill=blue!20] at (t3) [label=center:$v'$] {};

  \node (p1) at (0.5,  2.5) {};
  \node (p2) at (2.5,    2.5) {};
  \draw[thick, dotted, violet] (q1) -- node[lbl, left, violet]{$5$} (p1);
  \draw[thick, dotted, violet] (p1) -- node[lbl, above, violet]{$5$} (p2);
  \draw[thick, dotted, violet] (p2) -- node[lbl, right, violet]{$5$} (t3);

    \node[lbl] at (2, -1.2){$\Gamma^3_{v,v'}(G)$};
\end{scope}

\end{tikzpicture}
\caption[Path attachment creating a new cycle]{Path attachment $\Gamma^k_{v,v'}(G)$ with $k=3$. The vertices $v$ and $v'$
  are at distance $\graphdist_G(v,v') = 2$. Attaching a path of $3$ edges
  (dotted) creates a new cycle of length $3 + 2 = 5$.%
\label{fig:c5:path-attachment}}
\end{figure}

\begin{proposition}
\label{prop:c5:add-path}
Let $G \in \graphclass$ with $S = \egseq(G) \in \mathrm{Seq}(\mathcal{L}^*)$
non-empty, let $g_{\max} = \max(S)$, and let $v, v' \in \vertexset(G)$ and
$k \geq 1$. If $k \geq g_{\max} - \graphdist_G(v,v')$, then
$\Gamma^k_{v,v'}(G) \in \graphclass$ and
\[
  \egseq\!\left(\Gamma^k_{v,v'}(G)\right) = S \uplus \bigl((k + \graphdist_G(v,v'))^{(k)}\bigr),
\]
the $k$ added edges having edge-girth $k + \graphdist_G(v,v')$.
\end{proposition}

\begin{proof}
Set $G' = \Gamma^k_{v,v'}(G)$, $d = \graphdist_G(v,v')$, and let $P$ be the
added path, so that $k + d \geq g_{\max} \geq 3$. The graph $G'$ is connected,
and it is simple: if $k = 1$, then $d \geq g_{\max} - 1 \geq 2$, so that $v$
and $v'$ are distinct and not adjacent in $G$, and if $v = v'$, then
$k \geq g_{\max} \geq 3$.

Since the operation only adds edges, every cycle of $G$ is preserved in $G'$.
Let $C$ be a cycle of $G'$ that is not a cycle of $G$, that is, that uses an
edge of $P$. The internal vertices of $P$ are new, so their only edges are
those of $P$, and $C$ traverses all of $P$: it is $P$ itself when $v = v'$, and
otherwise it consists of $P$ together with a path of $G$ between $v$ and $v'$,
of length at least $d$. In both cases, $|C| \geq k + d \geq g_{\max}$. Hence
the new cycles have length at least $g_{\max} \geq g_{e,G}$ for every
$e \in \edgeset(G)$, and $g_{e, G'} = g_{e, G}$.

For the added edges, closing $P$ with a shortest path of $G$ between $v$ and
$v'$, or taking $P$ alone when $v = v'$, yields a cycle of length exactly
$k + d$, and every cycle containing an added edge has length at least $k + d$
by the argument above. Hence each added edge has edge-girth exactly $k + d$,
which gives $\egseq(G') = S \uplus ((k + d)^{(k)})$.
\end{proof}


\section{Base Cases of the Characterization}
\label{sec:c5:base}

This section establishes realizability for two special families of sequences,
which serve as base cases for the recursive characterization of
\Cref{sec:c5:induction}.
We first handle sequences containing infinite
edge-girth values, showing that they reduce to sequences in $\seqset^*$
via a finite--infinite decomposition.
We then fully characterize which
constant sequences in $\mathrm{Seq}(\mathcal{L}^*)$ are realizable.

\subsection{Edges Not Belonging to Any Cycle}

\begin{remark}
For every $m \in \mathbb{N}$, the sequence $(\infty^{(m)})$ is realizable:
the path graph $P_{m+1}$ on $m+1$ vertices has exactly $m$ edges, each with
edge-girth $\infty$.
\end{remark}

\begin{proposition}[Finite--infinite decomposition]
\label{prop:c5:decomposition}
Every $S \in \mathrm{Seq}(\mathcal{L})$ decomposes uniquely as
$S = S^* \uplus (\infty^{(m)})$ with $S^* \in \mathrm{Seq}(\mathcal{L}^*)$ and
$m \in \mathbb{N}$, and $S \in \seqset$ if and only if $S^* \in \seqset^*$.
\end{proposition}

\begin{proof}
The decomposition is unique: $m$ is the multiplicity of $\infty$ in $S$, and
$S^*$ is the subsequence of its finite values.

Let $S \in \seqset$. If $m = 0$, then $S^* = S$. Otherwise, let
$G \in \egseq^{-1}(S)$. Since $\max(S) = \infty$, the set $\edgeset_{\max}$ of
\Cref{prop:c5:remove-girth}(2) is the set of the bridges of $G$. Removing it
and rejoining the components by vertex identifications leaves the edge-girth
of every other edge unchanged, by \Cref{prop:c5:remove-girth}(2), and yields a
realization of $S^*$, which is the single-vertex graph when $S^*$ is empty.
Hence $S^* \in \seqset^*$.

Conversely, if $S^* \in \seqset^*$, then $S = S^* \uplus (\infty^{(m)}) \in \seqset$,
since $(\infty^{(m)}) \in \seqset$ and $\seqset$ is closed under $\uplus$
(\Cref{prop:c5:concat}).
\end{proof}

\subsection{Realizability of Constant Sequences}
\label{sec:c5:constant-real}

\begin{proposition}[Realizability of constant sequences]
\label{prop:c5:constant}
Let $g \in \mathcal{L}^*$ and $m \in \mathbb{N}$. Then $(g^{(m)}) \in \seqset^*$
if and only if
\[
  m = 0 \quad \text{or} \quad m = g \quad \text{or} \quad m \geq \lceil 3g/2 \rceil,
\]
where $(g^{(0)})$ denotes the empty sequence.
\end{proposition}

\begin{proof}
Let $\mathcal{M}_g = \{ m \in \mathbb{N} : (g^{(m)}) \in \seqset^* \}$.

\paragraph{Necessity.} We show $\bigl(\bigl\{1,\ldots,g-1\bigr\} \cup
\bigl\{g+1,\ldots,\lceil 3g/2\rceil - 1\bigr\}\bigr) \cap \mathcal{M}_g = \varnothing$.

\textit{Case $1 \leq m \leq g-1$.} Suppose $G \in \egseq^{-1}((g^{(m)}))$.
Every edge of $G$ has edge-girth $g$, so every edge belongs to a cycle of
length $g$, which requires at least $g$ edges. Hence $|\edgeset(G)| \geq g > m$,
a contradiction.

\textit{Case $g+1 \leq m \leq \lceil 3g/2 \rceil - 1$.} Suppose
$G \in \egseq^{-1}((g^{(m)}))$, and let $C_1$ be a realizing cycle of an edge
of $G$, which is a $g$-cycle. Since $m > g$, there exists an edge
$e \notin \edgeset(C_1)$; let $C_2$ be a realizing cycle for $e$. Set
$m' := |\edgeset(C_1) \cap \edgeset(C_2)|$. Since $e \in \edgeset(C_2)
\setminus \edgeset(C_1)$ and $g_e = g$, \Cref{prop:c5:overlap} gives
$m' \leq \lfloor g/2 \rfloor$. Counting the edges of $C_1 \cup C_2$,
\[
  m \;\geq\; |\edgeset(C_1) \cup \edgeset(C_2)| \;=\; 2g - m'
  \;\geq\; 2g - \lfloor g/2 \rfloor \;=\; \lceil 3g/2 \rceil,
\]
contradicting $m \leq \lceil 3g/2 \rceil - 1$.

\paragraph{Sufficiency.} We show $\{0\} \cup \{g\} \cup \{\lceil 3g/2 \rceil,
\lceil 3g/2 \rceil + 1, \ldots\} \subseteq \mathcal{M}_g$.

\textit{Base cases.} $0 \in \mathcal{M}_g$ by convention. For any $n \geq 1$,
connecting $n$ disjoint $g$-cycles by successive vertex identifications yields
a graph with edge-girth sequence $(g^{(ng)})$ by Proposition~\ref{prop:c5:concat},
so $ng \in \mathcal{M}_g$ for all $n \geq 1$.

\textit{Incrementing by $m'$.} Let $m \in \mathcal{M}_g \setminus \{0\}$ and
$m' \in \{\lceil g/2 \rceil, \ldots, g\}$, and let $G \in \egseq^{-1}((g^{(m)}))$.
Choose $v, v' \in \vertexset(G)$ with $\graphdist_G(v,v') = g - m' \in \{0, \ldots,
\lfloor g/2 \rfloor\}$: take $v = v'$ if $m' = g$, and otherwise two vertices
of a realizing cycle of an edge of $G$, which exist by
Proposition~\ref{prop:c5:realizing-shortest}(2). Since
$m' = g - \graphdist_G(v,v')$, Proposition~\ref{prop:c5:add-path} gives
$\egseq(\Gamma^{m'}_{v,v'}(G)) = (g^{(m + m')})$, so $m + m' \in \mathcal{M}_g$.

\textit{Reaching all $m \geq \lceil 3g/2 \rceil$.} Write $m = ng + s$ with
$n \geq 1$ and $0 \leq s \leq g-1$, and distinguish three cases.
If $s = 0$, then $m \in \mathcal{M}_g$ by the base case.
If $\lceil g/2 \rceil \leq s \leq g-1$, start from $ng \in \mathcal{M}_g$ and
apply the incrementing step once with $m' = s$.
If $1 \leq s < \lceil g/2 \rceil$, then $m \geq \lceil 3g/2 \rceil$ forces
$n \geq 2$. Start from $(n-1)g \in \mathcal{M}_g$ and apply the incrementing
step twice, with $m'_1 = \lceil g/2 \rceil$ and $m'_2 = \lfloor g/2 \rfloor
+ s$. Both increments are valid: $m'_1 \in \{\lceil g/2 \rceil, \ldots, g\}$
trivially, and since $1 \leq s < \lceil g/2 \rceil$, we have $\lceil g/2
\rceil \leq \lfloor g/2 \rfloor + 1 \leq m'_2 \leq g - 1$. This yields
$(n-1)g + m'_1 + m'_2 = ng + s = m \in \mathcal{M}_g$.
\end{proof}


\section{Characterization of Realizable Sequences}
\label{sec:c5:induction}

In this section, we establish the characterization of $\seqset^*$.
\Cref{sec:c5:reduction} introduces the sets $\mathcal{D}_S^{(g)}$ and reduces
the realizability of a sequence to that of a shorter prefix sequence.
\Cref{sec:c5:chaining} proves a chaining lemma, the key technical tool of the
section, which bounds a maximum diameter from below by reading a single cycle of a
realization, and \Cref{sec:c5:characterize-D} gives the explicit
characterization of $\mathcal{D}_S^{(g)}$ in terms of the maximum diameter.

\subsection{Reduction to \texorpdfstring{$\mathcal{D}_S^{(g)}$}{D}}
\label{sec:c5:reduction}

\begin{definition}[Appendable multiplicities]
\label{def:c5:D}
Let $S \in \seqset^*$ and $g > \max(S)$. We define the set of \emph{valid multiplicities} for appending edge-girth $g$ to $S$ as
\[
  \mathcal{D}_S^{(g)} := \bigl\{ m \in \mathbb{N}^* : S \uplus (g^{(m)}) \in \seqset^* \bigr\}.
\]
In other words, $m \in \mathcal{D}_S^{(g)}$ if and only if one can extend the
realizable sequence $S$ by appending $m$ edges of edge-girth $g$ and obtain a
new realizable sequence.
\end{definition}

\begin{proposition}
\label{prop:c5:induction}
Let $S \in \mathrm{Seq}(\mathcal{L}^*)$ be non-empty, let $g > \max(S)$, and $m > 0$.
Then
\[
  S \uplus (g^{(m)}) \in \seqset^* \iff S \in \seqset^* \text{ and } m \in \mathcal{D}_S^{(g)}.
\]
\end{proposition}

\begin{proof}
$(\Rightarrow)$ Let $S' = S \uplus (g^{(m)}) \in \seqset^*$ and
$G \in \egseq^{-1}(S')$. Remove from $G$ all edges $e$ with $g_e = g = \max(S')$
and rejoin the components by vertex identifications. By
Proposition~\ref{prop:c5:remove-girth}(2), the edge-girth of every remaining
edge is unchanged, which yields a graph $G^*$ with $\egseq(G^*) = S$, so
$S \in \seqset^*$. Since $S' \in \seqset^*$ and $S \in \seqset^*$, we have
$m \in \mathcal{D}_S^{(g)}$ by definition.

$(\Leftarrow)$ Immediate from the definition of $\mathcal{D}_S^{(g)}$.
\end{proof}

\begin{remark}
\label{rem:c5:D-known}
Let $S \in \seqset^*$ and $g > \max(S)$. By Proposition~\ref{prop:c5:constant},
$(g^{(m)})$ is realizable for $m \in \{0, g\}$ or $m \geq \lceil 3g/2 \rceil$.
Since $\seqset^*$ is closed under $\uplus$ (Proposition~\ref{prop:c5:concat}),
it follows that $S \uplus (g^{(m)}) \in \seqset^*$ for the same values of $m$,
hence
\[
  \{g\} \cup \{\, m \geq \lceil 3g/2 \rceil \,\} \;\subseteq\; \mathcal{D}_S^{(g)}.
\]
Characterizing $\mathcal{D}_S^{(g)}$ therefore reduces to determining
\[
  \mathcal{D}_{S,<}^{(g)} := \{1, \ldots, g-1\} \cap \mathcal{D}_S^{(g)},
  \qquad
  \mathcal{D}_{S,>}^{(g)} := \{g+1, \ldots, \lceil 3g/2 \rceil - 1\} \cap \mathcal{D}_S^{(g)}.
\]
\end{remark}

\subsection{A Chaining Lemma}
\label{sec:c5:chaining}

Suppose that $G$ realizes $S \uplus (g^{(m)})$, with $g > \max(S)$, and let $F$ be
its set of $m$ edges of edge-girth $g$. Removing $F$ leaves a realization of $S$,
possibly disconnected. A $g$-cycle $C$ of $G$ that uses some edges of $F$ falls
apart, in this removal, into pieces that remain in the realization of $S$: we call
them the \emph{arcs} of $C$, and their lengths add up to $g$ minus the number of
edges of $F$ on $C$. The lemma below says that these arcs can be laid end to end in
a realization of $S$, with no shortcut between them. A single cycle of $G$ thus
exhibits two vertices of some realization of $S$ at distance at least the total
length of its arcs, which bounds $d^*_S$ from below.

The set $F$ need not contain all the edges of edge-girth $g$, a generality that
\Cref{sec:c5:characterize-D} requires; the only property used is that removing $F$ does not
change the edge-girth of any other edge. When $g = \max(\egseq(G))$ and $F$
contains all the edges of edge-girth $g$, this property holds by
\Cref{prop:c5:remove-girth}(2).

\begin{lemma}[Chaining]
\label{lem:c5:chain}
Let $G \in \graphclass$ and $g \in \mathcal{L}^*$. Let $F \subseteq \edgeset(G)$ be a
set of edges of edge-girth $g$ such that every edge outside $F$ has a realizing
cycle disjoint from $F$, and let $T$ be the sequence of the edge-girths of the
edges outside $F$, so that $\egseq(G) = T \uplus (g^{(|F|)})$. Let $C$ be a
$g$-cycle of $G$ that contains $m'$ edges of $F$, with $1 \leq m' \leq g - 1$. Then
$T$ is realizable, and it has a realization in which two vertices are at distance
at least $g - m'$. In particular, $d^*_T \geq g - m'$.
\end{lemma}

\begin{figure}[htbp]
\centering
\begin{tikzpicture}[
  vtx/.style={circle, fill=black, inner sep=0pt, minimum size=4pt},
  lbl/.style={font=\small},
]
\def\R{2.2}
\draw[line width=1.6pt, blue!70!black] (100:\R) arc (100:160:\R);
\draw[line width=1.6pt, blue!70!black] (190:\R) arc (190:260:\R);
\draw[line width=1.6pt, green!45!black] (290:\R) arc (290:380:\R);
\draw[line width=1pt, dashed] (160:\R) arc (160:190:\R);
\draw[line width=1pt, dashed] (260:\R) arc (260:290:\R);
\draw[line width=1pt, dashed] (20:\R) arc (20:100:\R);
\foreach \a in {100,160,190,260,290,20} { \node[vtx] at (\a:\R) {}; }
\node[lbl] at (100:\R+0.35) {$x_1$};
\node[lbl] at (160:\R+0.35) {$y_1$};
\node[lbl] at (190:\R+0.35) {$x_2$};
\node[lbl] at (260:\R+0.35) {$y_2$};
\node[lbl] at (290:\R+0.35) {$x_3$};
\node[lbl] at (20:\R+0.35) {$y_3$};
\node[lbl, blue!70!black] at (130:\R-0.45) {$B_1$};
\node[lbl, blue!70!black] at (225:\R-0.45) {$B_2$};
\node[lbl, green!45!black] at (335:\R-0.45) {$B_3$};
\node[lbl] at (175:\R-0.45) {$A_1$};
\node[lbl] at (275:\R-0.45) {$A_2$};
\node[lbl] at (60:\R-0.45) {$A_3$};
\node[lbl] at (0,0) {$C$};
\end{tikzpicture}
\caption[Arcs and runs of a cycle in the chaining lemma]{The cycle $C$ in the proof of
\Cref{lem:c5:chain}, cut into arcs $B_j$ (solid, edges outside $F$) and runs $A_j$
(dashed, edges of $F$). Here $B_1$ and $B_2$ lie in the same component of
$G \setminus F$ (blue) and $B_3$ in another one (green). Since $B_2$ is the last
arc of the blue component, $y_2$ is identified with $x_3$. By
\eqref{eq:c5:chain-segment}, $\graphdist(x_1, y_2) \geq |B_1| + |B_2|$ and
$\graphdist(x_3, y_3) \geq |B_3|$, so that $x_1$ and $y_3$ end up at distance at
least $|B_1| + |B_2| + |B_3| = g - m'$.%
\label{fig:c5:chaining}}
\end{figure}

\begin{proof}
Let $H := G \setminus F$, obtained by removing all the edges of $F$, not only
those of $C$. Removing $F$ cuts $C$ into arcs of total length
$g - m'$. We first build a graph $G^*$ by rejoining the connected components of
$H$ so that these arcs line up, and then show that $G^*$ realizes $T$ and has
two vertices at distance at least $g - m'$.

\textit{Construction of $G^*$.} Since $1 \leq m' \leq g - 1$, the cycle $C$ has
edges in $F$ and edges outside $F$, so that, starting at a suitable vertex and
in a suitable direction, we can write
\[
  C \;=\; B_1\, A_1\, B_2\, A_2\, \cdots\, B_t\, A_t, \qquad t \geq 1,
\]
where each \emph{arc} $B_j$ is a path with all its edges outside $F$, and each
\emph{run} $A_j$ a path with all its edges in $F$
(\Cref{fig:c5:chaining}). The arcs have total length
$|B_1| + \cdots + |B_t| = g - m'$. Write $x_j$ and $y_j$ for the first and last
vertices of $B_j$. Each arc is a path of $H$, so it lies in a single component
of $H$.

The arcs may be spread over several components of $H$. We say that $B_i$ is the
\emph{last arc} of its component when no arc $B_j$ with $j > i$ lies in the
same component of $H$. The graph $G^*$ is obtained from $H$ by identifying
$y_i$ with $x_{i+1}$ for every $i \in \{1, \ldots, t - 1\}$ such that $B_i$ is
the last arc of its component, and then attaching every component of $H$ that
contains no arc to an arbitrary vertex. In other words, we follow $C$ and glue
the end of an arc to the start of the next one each time $C$ leaves a component
for good. In \Cref{fig:c5:chaining}, $B_2$ is the last arc of the blue
component, so $y_2$ is identified with $x_3$, whereas nothing is done at $y_1$,
since $C$ comes back to the blue component with $B_2$.

\textit{$G^*$ realizes $T$.} Each component of $H$ that contains an arc, except
that of $B_t$, is glued once, to a component whose last arc comes later, so
that the gluings close no cycle: $G^*$ is obtained by rejoining the connected
components of $H$ through successive vertex identifications. Since every edge
of $H$ has in $H$ the edge-girth it has in $G$, by hypothesis and
\Cref{prop:c5:remove-girth}(1), $G^*$ realizes $T$ (\Cref{prop:c5:concat}(1));
in particular, $T$ is realizable.

\textit{No shortcut along $C$.} Let $i \leq j$ be such that $B_i$ and $B_j$ lie
in the same component of $H$, let $Q$ be a shortest path between $x_i$ and
$y_j$ in $H$, and split $C$ into the part $R$ from $x_i$ to $y_j$, which
contains $B_i, \ldots, B_j$, and the rest $R'$, which starts with the run $A_j$
and thus contains an edge $e' \in F$. Following $Q$ and then $R'$ back is a
closed walk of $G$ that traverses $e'$ exactly once, $Q$ avoiding $F$, so that
$|Q| + |R'| \geq g_{e'} = g = |R| + |R'|$ by \Cref{lem:c5:closed-walk}(2), which
gives
\begin{equation}
\label{eq:c5:chain-segment}
  \graphdist_H(x_i, y_j) \;=\; |Q| \;\geq\; |R| \;\geq\; |B_i| + |B_{i+1}| + \cdots + |B_j|.
\end{equation}

\textit{Two vertices of $G^*$ far apart.} We show, by induction on $t - k$, that
\[
  \graphdist_{G^*}(x_k, y_t) \;\geq\; |B_k| + |B_{k+1}| + \cdots + |B_t|
  \qquad (1 \leq k \leq t),
\]
which for $k = 1$ gives $\graphdist_{G^*}(x_1, y_t) \geq g - m'$. Let $B_i$ be
the last arc of the component $K$ of $B_k$, so that $i \geq k$; distances
within $K$ are the same in $G^*$ as in $H$ (\Cref{prop:c5:concat}(2)). If
$i = t$, the inequality is~\eqref{eq:c5:chain-segment}. Otherwise, $K$ is glued
at $y_i = x_{i+1}$ to the component of $B_{i+1}$, and since the gluings close
no cycle, every path from $x_k$ to $y_t$ in $G^*$ goes through this vertex.
Hence, by~\eqref{eq:c5:chain-segment} and the induction hypothesis for $i + 1$,
\[
  \graphdist_{G^*}(x_k, y_t)
  \;=\; \graphdist_H(x_k, y_i) + \graphdist_{G^*}(x_{i+1}, y_t)
  \;\geq\; \bigl(|B_k| + \cdots + |B_i|\bigr) + \bigl(|B_{i+1}| + \cdots + |B_t|\bigr). \qedhere
\]
\end{proof}

\subsection{Characterizing \texorpdfstring{$\mathcal{D}_{S}^{(g)}$}{D}}
\label{sec:c5:characterize-D}

\begin{proposition}
\label{prop:c5:d-low-diam}
Let $S \in \seqset^*$ be non-empty and $g > \max(S)$. Then
\[
  \mathcal{D}_{S,<}^{(g)} = \{\max(1,\, g - d^*_S),\, \ldots,\, g-1\}.
\]
\end{proposition}

\begin{proof}
We proceed by double inclusion: the inclusion $(\supset)$ follows by explicit
path attachment, and $(\subset)$ by \Cref{lem:c5:chain}.

$(\supset)$ Let $G \in \egseq^{-1}(S)$ attain diameter $d^*_S$, and let
$m \in \{\max(1, g - d^*_S), \ldots, g-1\}$. Since $1 \leq g - m \leq d^*_S$, a
shortest path between two vertices at distance $d^*_S$ in $G$ contains two
vertices $v, v'$ with $\graphdist_G(v,v') = g - m$. As $g > \max(S)$, we have
$m \geq \max(S) - \graphdist_G(v,v')$, so Proposition~\ref{prop:c5:add-path}
yields
\[
  \egseq\!\left(\Gamma^m_{v,v'}(G)\right) = S \uplus (g^{(m)}),
\]
hence $m \in \mathcal{D}_{S,<}^{(g)}$.

$(\subset)$ Let $m \in \mathcal{D}_{S,<}^{(g)}$, let
$G \in \egseq^{-1}(S \uplus (g^{(m)}))$, and let $F$ be the set of its $m$ edges
of edge-girth $g$. Since $g = \max(\egseq(G))$, every edge outside $F$ has a
realizing cycle disjoint from $F$ (\Cref{prop:c5:remove-girth}(2)). Let $C$ be a
realizing cycle of an edge of $F$: it is a $g$-cycle with
$1 \leq |\edgeset(C) \cap F| \leq m \leq g - 1$. \Cref{lem:c5:chain}, in which
$T = S$, gives $d^*_S \geq g - |\edgeset(C) \cap F| \geq g - m$, that is,
$m \geq g - d^*_S$.
\end{proof}

Beyond $g$, the realizable multiplicities are governed by the maximum
diameters of the intermediate sequences $S \uplus (g^{(k)})$ with $k < g$. By
\Cref{prop:c5:d-low-diam}, these sequences are realizable exactly for
$\underline{m}_S^{(g)} \leq k \leq g - 1$, where
$\underline{m}_S^{(g)} := \max(1,\, g - d^*_S)$, and we set
\begin{equation}
\label{eq:c5:mbar}
  \overline{m}_S^{(g)} \;:=\; \max\Bigl(g + 1,\;
    \min_{\underline{m}_S^{(g)} \leq k \leq g-1}
    \bigl(k + g - d^*_{S \uplus (g^{(k)})}\bigr)\Bigr).
\end{equation}
The minimum is taken over a non-empty set: a non-empty $S \in \seqset^*$ has
$d^*_S \geq 1$, so that $\underline{m}_S^{(g)} \leq g - 1$.

\begin{proposition}
\label{prop:c5:d-high-diam}
Let $S \in \seqset^*$ be non-empty and $g > \max(S)$. Then
\[
  \mathcal{D}_{S,>}^{(g)}
  = \bigl\{\overline{m}_S^{(g)},\, \ldots,\, \lceil 3g/2 \rceil - 1\bigr\},
\]
and this set is non-empty.
\end{proposition}

\begin{proof}
Write $\underline{m} = \underline{m}_S^{(g)}$, $\overline{m} = \overline{m}_S^{(g)}$,
and $S_k := S \uplus (g^{(k)})$.

\textit{Non-emptiness.} First, $g + 1 \leq \lceil 3g/2 \rceil - 1$ for all
$g \geq 3$. Second, a realization of $S_{\underline{m}}$ contains a $g$-cycle,
so that $d^*_{S_{\underline{m}}} \geq \lfloor g/2 \rfloor$ by
\Cref{prop:c5:realizing-shortest}(2); the minimum in~\eqref{eq:c5:mbar} is
therefore at most $\underline{m} + g - \lfloor g/2 \rfloor \leq (g - 1) +
\lceil g/2 \rceil = \lceil 3g/2 \rceil - 1$.

We proceed by double inclusion. The inclusion $(\supset)$ follows by explicit
path attachment to a realization of some $S_k$, and $(\subset)$ by two
applications of \Cref{lem:c5:chain}.

$(\supset)$ Let $m \in \{\overline{m}, \ldots, \lceil 3g/2 \rceil - 1\}$, let
$k$ attain the minimum in~\eqref{eq:c5:mbar}, and set
\[
  k' := \max(k, m - g).
\]
We realize $S \uplus (g^{(m)})$ by attaching a path of $m - k'$ edges to an
optimal realization of $S_{k'}$: let $G \in \egseq^{-1}(S_{k'})$ attain
diameter $d^*_{S_{k'}}$, and let $v, v' \in \vertexset(G)$ be at distance
$g - (m - k') = k' + g - m$, with $v = v'$ when $k' = m - g$. As
$\max(S_{k'}) = g$, \Cref{prop:c5:add-path} yields
\[
  \egseq\!\left(\Gamma^{m - k'}_{v,v'}(G)\right) = S_{k'} \uplus (g^{(m - k')}) = S \uplus (g^{(m)}),
\]
so that $m \in \mathcal{D}_{S,>}^{(g)}$.

This construction requires $S_{k'}$ to be realizable and
$0 \leq k' + g - m \leq d^*_{S_{k'}}$, that is,
\[
  \underline{m} \leq k' \leq g - 1, \qquad
  k' \geq m - g, \qquad
  m \geq k' + g - d^*_{S_{k'}},
\]
which we now check. First, since
$m - g \leq \lceil 3g/2 \rceil - 1 - g \leq g - 1$, we have
$\underline{m} \leq k \leq k' \leq g - 1$, so that $S_{k'}$ is realizable by
\Cref{prop:c5:d-low-diam}; in particular, since $m \geq g + 1$, the path has
$m - k' \geq 2$ edges. Second, $k' \geq m - g$ by definition. Third, if
$k' = k$, the condition reads $m \geq k + g - d^*_{S_k}$, which holds since
$m \geq \overline{m}$; otherwise $k' = m - g$, and it reads
$d^*_{S_{k'}} \geq 0$.

$(\subset)$ Let $m \in \mathcal{D}_{S,>}^{(g)}$, let
$G \in \egseq^{-1}(S \uplus (g^{(m)}))$, and let $F$ be the set of its $m$
edges of edge-girth $g$. Every edge outside $F$ has a realizing cycle disjoint
from $F$ (\Cref{prop:c5:remove-girth}(2)). Since $m \geq g + 1$, it suffices
to find $k$ with $\underline{m} \leq k \leq g - 1$ and
$m \geq k + g - d^*_{S_k}$.

Let $C_1$ be a realizing cycle of an edge of $F$ that is not contained in $F$.
Such a cycle exists: otherwise, a realizing cycle $C$ of an edge of $F$ and a
realizing cycle $C'$ of an edge of $F$ outside $C$, which exists since $m > g$,
would both lie in $F$, and \Cref{prop:c5:overlap} would give
$m \geq |\edgeset(C) \cup \edgeset(C')| \geq 2g - \lfloor g/2 \rfloor =
\lceil 3g/2 \rceil$, a contradiction. Its number $k := |\edgeset(C_1) \cap F|$ of edges in $F$
thus satisfies $1 \leq k \leq g - 1$, and \Cref{lem:c5:chain}, applied to
$G$, $F$ and $C_1$, gives $k \geq g - d^*_S$, so that
$\underline{m} \leq k \leq g - 1$.

It remains to show that $m \geq k + g - d^*_{S_k}$. Since $m > g > k$, some
edge $e' \in F$ lies outside $C_1$; let $C_2$ be a realizing cycle of $e'$,
let $F' := (F \cap \edgeset(C_2)) \setminus \edgeset(C_1)$ be its edges of $F$
outside $C_1$, and $m' := |F'|$, so that $1 \leq m' \leq g$ and
$k + m' \leq m$. It suffices to show that $m' \geq g - d^*_{S_k}$, which is
clear if $m' = g$; we assume that $m' \leq g - 1$.

Let $G''$ be obtained from $G$ by removing
$F \setminus (\edgeset(C_1) \cup \edgeset(C_2))$ and rejoining the components
by vertex identifications. By \Cref{prop:c5:remove-girth}(1), the edges
outside $F$ keep their edge-girths, and so do the edges of $F$ that lie on
$C_1$ or on $C_2$, since these cycles survive; hence
$\egseq(G'') = S \uplus (g^{(k + m')})$. The edges outside $F'$ form the
sequence $S_k$, and each of them has a realizing cycle disjoint from $F'$: for
an edge outside $F$, this was noted above, and for an edge of
$F \cap \edgeset(C_1)$, the cycle $C_1$ avoids $F'$. \Cref{lem:c5:chain},
applied to $G''$, $F'$ and $C_2$, now gives $m' \geq g - d^*_{S_k}$, which
concludes the proof.
\end{proof}

\section{Main Result}
\label{sec:c5:main}

\begin{theorem}[Characterization of realizable edge-girth sequences]
\label{thm:c5:main}
Let $S \in \mathrm{Seq}(\mathcal{L})$ be non-empty.
\begin{enumerate}
  \item[\textup{(i)}] If $S = (g^{(m)})$ is a constant sequence, then $S$ is
    realizable if and only if
    \[
      g = \infty \quad \text{or} \quad m = g \quad \text{or} \quad m \geq \lceil 3g/2 \rceil.
    \]
  \item[\textup{(ii)}] Otherwise, write $S = S_0 \uplus (g^{(m)})$, where
    $g = \max(S)$, $m \geq 1$ is the multiplicity of $g$ in $S$, and $S_0$ is
    non-empty. If $g = \infty$, then $S$ is realizable if and only if $S_0$ is
    realizable. If $g < \infty$, then $S$ is realizable if and only if $S_0$ is
    realizable and
    \[
      \underline{m} \leq m \leq g \quad \text{or} \quad m \geq \overline{m},
    \]
    where, with $d^*_T$ the maximum diameter of the realizations of a
    sequence $T$ (\Cref{def:c5:diam}),
    \[
      \underline{m} \;:=\; \max(1,\, g - d^*_{S_0}), \qquad
      \overline{m} \;:=\; \max\Bigl(g + 1,\;
        \min_{\underline{m} \leq k \leq g-1}
        \bigl(k + g - d^*_{S_0 \uplus (g^{(k)})}\bigr)\Bigr),
    \]
    the sequences $S_0 \uplus (g^{(k)})$, $\underline{m} \leq k \leq g - 1$,
    being realizable whenever $S_0$ is.
\end{enumerate}
\end{theorem}

\begin{proof}
Part~\textup{(i)}: for $g = \infty$, the sequence is realized by a path
(\Cref{sec:c5:base}); for $g < \infty$, this is
Proposition~\ref{prop:c5:constant}.

Part~\textup{(ii)}: if $g = \infty$, then $S_0$ is the subsequence of finite
values of $S$, and Proposition~\ref{prop:c5:decomposition} gives the
equivalence. If $g < \infty$, Proposition~\ref{prop:c5:induction} gives
$S \in \seqset^*$ if and only if $S_0 \in \seqset^*$ and
$m \in \mathcal{D}_{S_0}^{(g)}$. With $\underline{m} = \underline{m}_{S_0}^{(g)}$
and $\overline{m} = \overline{m}_{S_0}^{(g)}$, Remark~\ref{rem:c5:D-known} and
Propositions~\ref{prop:c5:d-low-diam} and~\ref{prop:c5:d-high-diam} give
\[
  \mathcal{D}_{S_0}^{(g)} \;=\; \{\underline{m}, \ldots, g - 1\} \cup \{g\}
  \cup \{\overline{m}, \ldots, \lceil 3g/2 \rceil - 1\}
  \cup \{\, m \geq \lceil 3g/2 \rceil \,\}
  \;=\; \{\underline{m}, \ldots, g\} \cup \{\, m \geq \overline{m} \,\},
\]
the last equality because $\overline{m} \leq \lceil 3g/2 \rceil - 1$
(Proposition~\ref{prop:c5:d-high-diam}). The realizability of
$S_0 \uplus (g^{(k)})$ for $\underline{m} \leq k \leq g - 1$ is
Proposition~\ref{prop:c5:d-low-diam}.
\end{proof}


\section{Concluding Remarks}
\label{sec:c5:conclusion}

In this paper, we introduced the edge-girth sequence $\sigma(G)$ of a simple
connected graph $G$ as the nondecreasing sequence of edge-girth values over
all its edges, and asked which sequences are realizable. Our main result,
\Cref{thm:c5:main}, gives a complete recursive characterization: a
sequence $S = S_0 \uplus (g^{(m)})$ is realizable if and only if $S_0$ is
realizable and the multiplicity $m$ lies in a set of admissible values
determined by the maximum diameters of $S_0$ and of its extensions
$S_0 \uplus (g^{(k)})$, $k < g$, where $g = \max(S)$. Several questions
remain open.

\paragraph{Computing the maximum diameter.}
\Cref{thm:c5:main} makes realizability depend on the maximum diameters
$d^*_{S_0 \uplus (g^{(k)})}$, $0 \leq k < g$, and it is effective only
insofar as these quantities can be computed. They can be obtained by
exhaustive search for small sequences, but no efficient algorithm and no
closed form are known in general. A natural route is to reduce $d^*_S$ to
the 2-connected graphs whose blocks make up a realization, and to study the
extremal problem it leads to: the largest diameter of a 2-connected graph
with a given number of edges, every edge of which lies on a short cycle. The
same question governs whether the minimum defining $\overline{m}$ admits a
simpler expression.

\paragraph{Planarity.}
The constructions of this paper do not preserve planarity in general: a path
attachment between two vertices that share no face in any plane embedding
produces a non-planar graph. Characterizing the sequences that a planar graph can
realize, and more generally which additional graph properties are compatible
with the realizability criterion of \Cref{thm:c5:main}, is a natural
direction for future work.

\paragraph{Realizable numbers of vertices.}
The realization problem studied here imposes constraints only on the edges of
$G$: two graphs may realize the same edge-girth sequence while having
different numbers of vertices (\Cref{fig:c5:same-seq-diff-vertices}). For
a given realizable sequence $S$, a natural question is to characterize the
set of integers $n$ such that some graph $G \in \sigma^{-1}(S)$ has
$|V(G)| = n$. More generally, combining edge-based constraints such as the
edge-girth sequence with vertex-based constraints such as degree sequences
would yield a richer class of realization problems that more tightly
prescribes the local structure of the graph.

\begin{figure}[htbp]
\centering
  \begin{tikzpicture}
    \node[draw, circle] (A) at (0, 1) {};
    \node[draw, circle] (B) at (-1, 0) {};
    \node[draw, circle] (C) at (1, 0)  {};
    \node[draw, circle] (D) at (0, 2)  {};
    \node[draw, circle] (E) at (0, 3)  {};
    \draw[blue] (A) -- (B) -- (C) -- (A);
    \draw[blue] (A) -- (D) -- (E);
    \draw[blue] (B) -- (D);
    \draw[blue] (B) -- (E);
    \draw[blue] (C) -- (D);
    \draw[blue] (C) -- (E);
  \end{tikzpicture}
  \hspace{2cm}
  \begin{tikzpicture}
    \node[draw, circle] (X) at (0, 1.5) {};
    \node[draw, circle] (Y) at (-1, 0) {};
    \node[draw, circle] (Z) at (1, 0)  {};
    \node[draw, circle] (Z2) at (2, 1.5)  {};
    \node[draw, circle] (Z3) at (3, 0)  {};
    \node[draw, circle] (Z4) at (4, 1.5)  {};
    \node[draw, circle] (Z5) at (5, 0)  {};
    \draw[blue] (X) -- (Y) -- (Z) -- (X);
    \draw[blue] (Z) -- (Z3) -- (Z2) -- (Z);
    \draw[blue] (Z3) -- (Z4) -- (Z5) -- (Z3);
  \end{tikzpicture}
\caption{Two graphs realizing the same edge-girth sequence $\sigma(G) =
(3^{(9)})$ with different numbers of vertices: $5$ vertices (left) and $7$
vertices (right).%
\label{fig:c5:same-seq-diff-vertices}}
\end{figure}
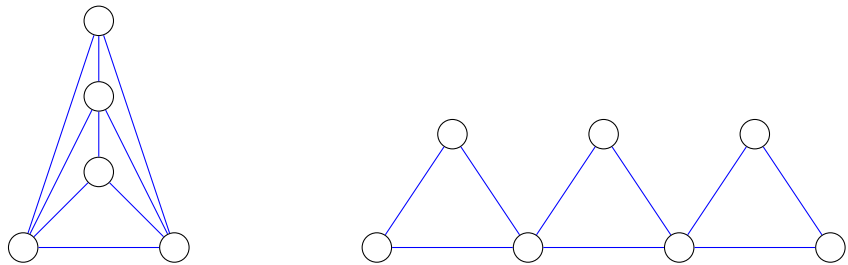

\paragraph{Counting realizing graphs.}
The constructions of this paper exhibit a realization of every realizable
sequence. A natural follow-up is to count the non-isomorphic graphs realizing
a sequence $S$, in the spirit of classical enumeration results for degree
sequences~\cite{bender1978asymptotic}.

\section*{Acknowledgements}
The main proof strategy and first draft were developed by L.M.;
P.H. and C.L. contributed to the revision and final presentation.

\bibliographystyle{plainnat}
\bibliography{main}

\end{document}